\documentclass[11pt]{amsart}

\RequirePackage[l2tabu,orthodox]{nag}
\usepackage[T1]{fontenc}
\usepackage[utf8]{inputenc}
\usepackage{lmodern}
\usepackage[english]{babel}
\usepackage{color} 
\usepackage[usenames,dvipsnames]{xcolor}
\usepackage{amsrefs,amsthm,amssymb,amsmath}
\usepackage[pdftex]{graphicx}
\usepackage{enumerate}
\usepackage[inline]{enumitem}
\usepackage{url}
\usepackage{stmaryrd}

\usepackage{multicol}
\usepackage{mathrsfs}

\usepackage[foot]{amsaddr}

\numberwithin{equation}{section}

\let\oldtocsection=\tocsection

\let\oldtocsubsection=\tocsubsection 

\let\oldtocsubsubsection=\tocsubsubsection

\renewcommand{\tocsection}[2]{\hspace{0em}\oldtocsection{#1}{#2}}
\renewcommand{\tocsubsection}[2]{\hspace{2em}\oldtocsubsection{#1}{#2}}
\renewcommand{\tocsubsubsection}[2]{\hspace{4em}\oldtocsubsubsection{#1}{#2}}

\newtheorem{thm}{Theorem}[section]

\newtheorem{lem}[thm]{Lemma}
\newtheorem{prop}[thm]{Proposition}
\newtheorem{cor}[thm]{Corollary}
\newtheorem{intro-cor}[introthm]{Corollary}

\theoremstyle{definition}
\newtheorem{dfn}[thm]{Definition}

\newtheorem*{main question}{Main Question}

\theoremstyle{remark}
\newtheorem{rem}[thm]{Remark}

\usepackage[margin=3cm]{geometry}

\usepackage{indentfirst}
\usepackage{microtype}

\usepackage[colorlinks=true,linkcolor=blue,citecolor=magenta]{hyperref}

\usepackage{blindtext}

\usepackage{nameref}

\DeclareMathOperator{\homeo}{\mathsf{Homeo}}

\newcommand{\vs}{\vspace{0,5cm}}
\newcommand{\vsc}{\vspace{0,3cm}}

\newcommand{\id}{\mathsf{id}}
\newcommand{\R}{\mathbb{R}}

\newcommand{\Z}{\mathbb Z}

\renewcommand{\setminus}{\smallsetminus}

\renewcommand{\emptyset}{\varnothing}

\title{Free product of bi-orderable group is bi-orderable: A simple dynamical proof}
\author{Juan Alonso and Crist\'obal Rivas}
\date{\today}

\begin{document}

\begin{abstract} 
We give a new  proof of Vinogradov's theorem asserting   that free product of  biorderable groups is biorderable. Our proof relies on a simple dynamical construction that provides a bi-ordering of the free product that extends the bi-orderings on the factors. Furthermore, if the bi-orders of the factors are invariant under some automorphisms, then the bi-ordering on the product turns out to be invariant under the product automorphism.

\end{abstract}

\maketitle

\section{Introduction}

We say that a group is {\em left-orderable} if it supports a total ordering which is invariant under left   multiplication and {\em bi-orderable} if the ordering is also invariant under right multiplication (see \S \ref{preliminars} for precisions). We use the term {\em orderable} when there is no harm in the ambiguity. In 1949, Vinogradov \cite{vino} used groups of positive units in ordered rings to show  that the free product of two bi-ordered groups admits a bi-ordering  that extends the bi-orders of the factors. This approach was latter simplified Johnson \cite{johnson} in 1972, and even further   by Bergman \cite{bergman} in 1990 who considered specific rings of matrices over groups-rings (see also \cite[\S2.1.2]{GOD}). 
Recently, Muliarchyk \cite{muliarchyk2} obtained the same result but with a completely different method which considers group actions by homeomorphisms of the real line. Unfortunately, Muliarchyk's approach dependes heavily on the intricate embedding of a (countable) bi-ordered group into the group of orientation preserving homeomorphisms of the real line  build by Dovhyi and Muliarchyk in \cite{muliarchyk1}, together with a delicate perturbation argument to show that the bi-ordering obtained is in fact an ordering on the free product and not only on a quotient of it.

The aim of this note is to give another proof of Vinogradov's theorem. Our approach is also dynamical, yet it uses the usual  {\em dynamical realization} of an ordered group  as a group of homeomorphisms of the line (see, for instance, \cite{BMRT,clay-rolfsen,GOD,ghys,navas}, or \S \ref{sec dynamical realization}) and   (modifications of) ping-pong partitions to ensure faithfulness of the representation. Furthermore, with our construction we get that  if the bi-orders of the factors are invariant under some automorphisms, then the bi-ordering on the product turns out to be invariant under the product automorphisms, thus recovering a result from Rolfsen \cite{rolfsen}.

\begin{thm} \label{teo main} If $(A,\preceq_A)$ and $(B,\preceq_B)$ are  biordered groups, then  $A*B$ admits a biorder $\preceq^*$ that extends $\preceq_A$ and $\preceq_B$.  Moreover if $f:A\to A$ and $g:B\to B$ are automorphisms that preserves the bi-orderings $\preceq_A$ and $\preceq_B$, then $\preceq^*$  is invariant under the induced automorphism $f*g$.

\end{thm}

\begin{rem} As pointed out by Passman \cite{passman} Vinogradov's approach also shows that the free product of left-orderable groups is left-orderable. This is also the case for our approach as we point out in Corollary \ref{coro left-orders}.  In \cite{dicks-sunic} Dicks and Šunić  describe a very nice geometrical method to produce left-orderings on free product of groups out of actions on oriented trees.

\end{rem}




Our proof of Theorem \ref{teo main} is organized in three steps. We first show in \S \ref{sec countable} that the free product $A*B$ of two {\em countable} bi-ordered groups admits a bi-order that extends the bi-orders on the factors. For this we  modify a natural ping-pong action of $A*B$ on the line to produce an action on the real line by  homeomorphisms whose supports are bounded from below, and with the property that each homeomorphism has a well defined sign at the infimum of its support (see Proposition \ref{prop principal}). This behaviour is much like that of piecewise analytic  homeomorphisms with support bounded from below,  compare with  \cite{muliarchyk1,muliarchyk2}. Then, in \S \ref{sec invariance}, we exploit the fact that an order preserving automorphisms of a countable ordered group can always be implemented by conjugation of its corresponding dynamical realizations (see Proposition \ref{prop auto es conjugar}) to produce a conjugation of our special action of $A*B$ that implements our given automorphism $f*g$ from Theorem \ref{teo main}. This allows us to finish the proof in the countable case. Finally,  in \S \ref{sec uncountable}, we prove Theorem \ref{teo invariant Burns Hale} which is a variation of  Burns and Hale theorem \cite{burns-hale} that enable us to pass from the countable case of Theorem \ref{teo main} to the general one.

\begin{rem} The construction of our special action of $A*B$ on the line is much inspired in \cite[\S1.2.3]{GOD}, where an action of the free group $F_2$ on the line is constructed aiming  to show its bi-orderability. The main difference between both constructions, is that for the case of $F_2$ one can assume that the action is given by piecewise linear automorphism,  and hence the elements of $F_2$ certainly have a well defined sign around its infimum of support.  In our case, we cannot assume that the initial actions of $A$ and $B$ have any degree of regularity and so we have to elaborate a little bit in order to prove Proposition \ref{prop principal}.

\end{rem}

\paragraph{\bf Acknowledgements}  We are grateful to Joaqu\'in Brum for several discussions around orderability and for pointing out and explain  to us the work of Muliarchyk \cite{muliarchyk2}.  Both author acknowledge the support of Mathamsud {\em Interactions between groups and dynamics}.   C.R. also acknowledge the support of Fondecyt 1241135.

\section{Preliminars}
\label{preliminars}

Let us fix some basic notations. Throughout this work, if $X\subset\R$ we denote by $int(X)$ its interior and $cl(X)$ its closure. The identity of group will be denoted by $\id$.



\subsection{Bi-orders, partial bi-orders and positive cones} An order relation $\preceq$ (not necessarily total) on a group
$G$ is said to be a {\em partial bi-ordering} if for every
$g_1, g_2,g_3$ in $G$, we have that
$g_1\prec g_2$ implies $g_3g_1\prec
g_3g_2$ and $g_1g_3\prec
g_2g_3$.  An element $g \in G$ is called
$\preceq$-positive  if $\id\prec g$, where $\id$ denotes the identity of $G$. The subset of $\preceq$-positive
elements, usually called the {\em positive cone} for $\preceq$, will
be denoted by $P$. Clearly $P$ satisfies

\noindent $(o1)$ $P  P \subseteq P \;$ and $gP g^{-1}=P $ for every $g\in G$; that
is, $P$ is a normal sub semigroup,

\noindent $(o2)$ $P \cap P^{-1}=\emptyset$, where
$P^{-1}=\{g^{-1}\in G\mid g\in P\}=\{g\in
G\mid g\prec id\}$.

If in addition $\preceq$ is a total order, we will simply say that
$\preceq$ is a {\em bi-ordering}. In this case, the set of
$\preceq$-positive elements also satisfies

\noindent $(o3)$ $G=P \cup P^{-1} \cup \{\id\}$.

Conversely, given any subset $P\subseteq \Gamma$ satisfying the
conditions $(o1)$, $(o2)$ and $(o3)$ (resp. $(o1)$ and $(o2)$)
above, we can define a bi-ordering (resp. a partial bi-ordering)
$\preceq$ by letting $f\prec g$ if and only if $f^{-1}g\in P$.
We will usually identify $\preceq$ with $P$.

\subsection{Compactness of the space of partial bi-orders} \label{sec PBO}Given a group $G$ (of arbitrary cardinality), we denote the set
of all partial bi-orderings on $G$ by
$\mathcal{PBO}(G)$. This set has a natural topology first
exploited by Sikora for the case of (total orderings on) countable
groups \cite{sikora}. This topology can be defined by identifying
$P\in \mathcal{PBO}(G)$ with its characteristic function
$\chi_P \in \{0,1\}^G $. In this way, we can view
$\mathcal{PBO}(G)$ as embedded in $\{0,1\}^G$. This latter
space, with the product topology, is a Hausdorff, totally
disconnected, and compact space. It is not hard to see that (the
image of) $\mathcal{PBO}(G)$ is closed, and hence
compact as well (see \cite{navas,rivas free,GOD,clay-rolfsen} for details).
In the same way, for a bi-orderable group $G$, the space of
all bi-orderings, here denoted $\mathcal{BO}(G)$, is closed
in $\mathcal{PBO}(G)$, hence compact as well.

\subsection{Left-orderings and dynamical realization} \label{sec dynamical realization} A total ordering $\preceq$ on a group $G$ is called a left-orderings if $f\preceq g$ implies that $hf\preceq hg$ for all $f,g,h\in G$.  It is a classical fact that a countable group is left orderable if and only if it embeds into $\homeo_+(\R)$, the group of orientation preserving homeomorphisms of the real line (see  \cite{GOD,ghys,navas}). 

One direction of the  proof consists in building a {\em dynamical realization} of a countable  left-ordered  group.  To explain this let $(G,\preceq)$ be a countable and left-ordered group and  consider, first,  an order preserving injection $i:(G,\preceq)\to (\R,\leq)$, that is well behaved  in the sense that 
\begin{enumerate}
\item The image $i(G)$ is unbounded in both directions of the line,
\item Every connected component $I$ of the complement of $cl(i(G))$ satisfies that its boundary $\partial I\subset i(G)$.
\end{enumerate}
Concrete examples of well behaved  injections $i$ can be easily be  build algorithmically (see for instance \cite{ghys,navas}), but it requires $G$ to be countable (for instance $\omega_1$, the first uncountable ordinal, cannot be orderly embedded into $\R)$.  Other than  that any  bijection  $h:i(G)\to i(G)$ that preserves the order $\leq$  can be extended to an increasing homeomorphisms of the real line by first extending it to the closure of $i(G)$ and then taking linear interpolation on the complements. In fact, it is the whole group of order preserving bijections of $i(G)$ that extends to a group of homeomorphisms.  This is implicit in \cite{ghys,navas,GOD,clay-rolfsen} and is explicit in \cite[Lemma 2.2.14]{BMRT}.  In particular, since  $G$ induces an order preserving bijection of $i(G)$ via
$$g.(i(x))=i(g\cdot x),$$
we obtain an homomorphic embedding  $\rho:G\to \homeo_+(\R)$, that represents $G$ as a group of order preserving homeomorphisms of the line.  This is what we call a {\em dynamical realization} of $(G,\preceq)$.   We say that a point $x_0\in \R$ is a {\em reference point} of the dynamical realization if it satisfies that
$$id\prec g \Leftrightarrow x_0 < \rho(g) (x_0).$$
Reference points may not be unique but certainly $x_0=i(\id)$ is a reference point.

Analogously, any order-preserving automorphism $\varphi\in Aut(G,\preceq)$ defines an order preserving bijection of $i(G)$ via
$\varphi(i(x))=i(\varphi(x))$. Hence we obtain 
\begin{prop} \label{prop auto es conjugar}For every order-preserving automorfism of a countable left-ordered group  $\varphi\in Aut((G,\preceq))$, there is an increasing homeomorphism $h:\R\to\R$ such that 
$$h\circ \rho(g)\circ h^{-1}(x)=\rho(\varphi(g))(x)$$
for every $x\in \R$ and every $g\in G$.  

\end{prop}

\begin{rem} \label{rem free dynamical real}  In a dynamical realization of $\rho:G\to \homeo_+(\R)$, the set of fixed points of a non-trivial $g\in G$ has empty interior. Indeed    any point in the image of $i(G)$ has free orbit under the action of $\rho(G)$, and, although  fixed points may appear when passing to the action in its closure\footnote{In fact, an important result of H\"older \cite{holder} states that a group acting freely by homeomorphisms of the line must be abelian. See, for instance,  \cite{GOD} for a modern proof.},  outside this closure the $\rho$-action is also free. To see this note that for $I$ a (open) complementary interval of $\R\setminus cl(i(G))$, we have that $\partial I\subseteq i(G)$ and hence $I$ is a {\em wandering} interval for $\rho(G)$, meaning that $\rho(g)(I)\cap I=\emptyset $  for any $g\neq \id$.

\end{rem}


\begin{rem}\label{rem dynamical real for bi-orders} In the special case where  $\preceq$ is a bi-ordering of $G$, then  $\id\prec g$ implies that $x\prec g\cdot x $ for every $x\in G$, and hence every $\preceq$-positive element $g\in G$ satisfies that
$$x\leq \rho(g)(x)$$
for every $x\in \R$. See for instance \cite{navas}. We will use this last fact in the construction of our bi-order of $A*B$.
\end{rem}

\subsection{Free product and ping pong configurations} \label{sec ping pong}  A ping pong configuration is a simple dynamical mechanism that produces a faithful action of a free product of groups. Early applications of this mechanism can be traced back to Klein and Tits.  See \cite[\S II]{harpe} and the references therein. For our purpose we use the following formalization.

Let $\phi:G\to Sym(X)$ be an action of a group $G$ by bijections of a set $X$, and let $A$ and $B$ be subgroups of $G$. We say that a  {\em  ping pong configuration}  for the action of $\phi(A)$ and $\phi(B)$ is a collection of disjoint subsets $X_A$ and $X_B$ of $X$ such that  $X\setminus (X_A\cup X_B) \neq\emptyset$ and 
$$\phi(a)(X\setminus X_A)\subseteq X_A \text{ and } \phi(b)(X\setminus X_B)\subseteq X_B$$
for every $a\in A\setminus \{\id\}$ and every $b\in B\setminus \{\id\}$. 

\begin{lem}  \label{lem ping pong} With the notations above, if $X_A$ and $X_B$ is a ping pong configuration for the action of $\phi(A)$ and $\phi(B)$, then $\langle \phi(A),\phi(B)\rangle$, the group generated by $\phi(A)$ and $\phi(B)$, is isomorphic to the free product $A*B$. Moreover, any point in $X\setminus (X_A\cup X_B) \neq\emptyset$ has free orbit under $\langle \phi(A),\phi(B)\rangle$.

\end{lem}

\proof Clearly the group $\langle \phi(A),\phi(B)\rangle$ is an homomorphic image of the free product $A*B$.  Furthermore, the image of any point $x\in X\setminus (X_A\cup X_B)$  under $\phi(w)$  belongs to $X_A\cup X_B$ for any non-trivial  $w\in A*B$. Hence, the natural homomorphisms $\phi:A*B\to \langle\phi(A),\phi(B)\rangle$ is an isomorphism, and the points in $X\setminus (X_A\cup X_B)$ have free orbit.  $\hfill\Box$

\section{Proof in the countable case}\label{sec countable}

In this section we show the countable version of our main result.

\begin{thm} \label{teo main 2} Let $(A,\preceq_A)$ and $(B,\preceq_B)$ be  countable biordered groups.  Then
\begin{enumerate}
\item[$i)$]   The free product $A*B$ admits a biorder $\preceq^*$ that extends $\preceq_A$ and $\preceq_B$.
\item[$ii)$] If $f:A\to A$ and $g:B\to B$ are automorphisms that preserves the bi-orderings  $\preceq_A$ and $\preceq_B$, then $\preceq^*$  is invariant under the induced automorphism $f*g$.

\end{enumerate}
  
\end{thm}

Our proof consists in building a special action of the free product $A*B$ on the real line. Its construction requires three steps. 

Let $\preceq_A$ and $\preceq_B$ be biorders of $A$ and $B$ respectively.  

\vsc
\noindent \underbar {Step 1:} {\em Preparing the factors: a blow-up on the dynamical realizations}. 

\vsc

For convenience  we will  consider dynamical realizations of $(A,\preceq_A)$ and $(B,\preceq_B)$ to be inside $\homeo_+([0,1])$ instead of $\homeo_+(\R)$.  This can be done since the interval $(0,1)$ is orderly-homeomorphic to $\R$ and  every increasing homeomorphism of $\R$ fixes $+\infty$ and $-\infty$.

So let  $\rho_A:A\to Homeo_+([0,1])$ be a dynamical realization of $(A,\preceq)$ and assume that $x_A=1/2$ is a reference point.  We now make a {\bf blow-up} of the point $x_A$ to create a {\em new} interval on which $A$ still acts. For this we  first replace each point $o$ in the $A$-orbit of $x_A$  by a bounded  and closed  interval $I_o$. Since the orbit  is countable we can choose the length of the intervals $I_o$ to ensure that  their reunion has finite Lebesgue measure,  so that we end up with a {\em new} compact interval $J$.  Denote $J'$ the set of points of $J$ that are complementary to the reunion of $\{I_o\}$.  Note that on  $J'$  the group $A$ acts via $\rho_A$ and this action preserves the order of $J'$ coming from the line.    In particular (following \S\ref{sec dynamical realization})  this $A$ action on $J'$ can be extended to an action by homeomorphisms of the whole interval $J$. 

We call $\rho_A^{blow}:A\to \homeo_+(J)$  the resulting representation of $A$. After conjugation we can (and will) assume that $\rho_A^{blow}$  takes $A$ into $Homeo_+([0,1])$ and that the interval $I_{x_A}=[1/5,4/5]$. Remark that for 
$$\rho_A^{blow}:A\to \homeo_+([0,1])$$ the interior of the interval $I_{x_A}$ is {\em wandering}, meaning  that  $int(I _{x_A})$ is moved disjointly by any non-trivial element $a\in A$. Furthermore, any point $x\in int(I_{x_A})$ is a reference point for $\preceq$ in the sense that $\id \prec a$ if and only if $x< \rho_A^{blow}(a)(x)$.

For $(B,\preceq_B)$ we do the analogous construction to end up with $\rho_B^{blow}:B\to Homeo_+([0,1])$ and $int(I_{x_B})=(1/5,4/5)$ a wandering interval  made of reference points of $\preceq_B$.



\begin{rem} \label{rem extra wandering} In the representation $\rho_A^{blow}:A\to \homeo_+([0,1])$, the set of fixed points of a non-trivial $a\in A$ has empty interior. Indeed, from Remark \ref{rem free dynamical real} we know this to be true for the original dynamical realization $\rho_A$, and the procedure of blowing up only introduces wandering intervals, namely the intervals in the orbit of $int(I_{x_A})$.  The same holds for the representation $\rho_B^{blow}:B\to \homeo_+([0,1])$.


\end{rem}

\begin{rem} \label{rem automorfismo blow} Proposition \ref{prop auto es conjugar} also holds when the action is a blow-up of a dynamical realization. In fact we can define $h_A\in \homeo_+([0,1])$ so that $$h_A\circ \rho_A^{blow}(a)\circ h_A^{-1}(x)=\rho_A^{blow}(\varphi(a))(x)$$
and $h_A$ maps $\rho_A^{blow}(a)I_{x_A}\to \rho_A^{blow}(\varphi(a))I_{x_A}$ by an affine map for every $a\in A$. In particular $h_A$ is the identity on $I_{x_A}$. The same can be done for $B$.
\end{rem}



\vsc
\noindent \underbar {Step 2:} {\em Building a ping-pong.} 

\vsc

For $r$ a real number let $T_r:x\mapsto x+r$ be the traslation by $r$. We take $\widetilde\rho_A:A\to Homeo_+(\R)$ so that the restriction of $\widetilde\rho_A$ to $[0,1]$ equals $\rho_A^{blow}$ and, for each $n\in \Z$, the restriction of $\widetilde \rho_A$ to $[n,n+1]$ equals $T_n \circ \rho_A^{blow}\circ T_n^{-1}$. In other words $\widetilde\rho_A$ is the periodic repetition of $\rho_A^{blow}$. For $\widetilde\rho_B$ we do similarly, namely $\widetilde\rho_B:B\to Homeo_+(\R)$ is build such that on the interval $[n, n+1]+1/2$ the representation $\widetilde\rho_B$ equals $T_{n+1/2} \circ \rho_B^{blow} \circ T_{n+1/2}^{-1}$. We call $$\widetilde  \rho:A*B\to Homeo_+(\R)$$ the induced action of $A*B$ from $\widetilde\rho_A$ and $\widetilde\rho_B$.   In  Proposition \ref{prop ping pong} below, we show that $\widetilde\rho$ is  a faithful representation by exhibiting a {\em ping pong} configuration for the action. For its proof (and statement) we introduce some notations. These notations will be also needed in the  proof of Lemma \ref{lem functor}. 

For $j\in\Z$ let $I_A^j=(1/5,4/5)+j$ and $I_B^j=(1/5,4/5)+j+1/2$. These are  wandering intervals under $\widetilde{\rho}(A)$ and $\widetilde{\rho}(B)$ respectively. Put 
$$\mathcal{I}_A=\bigcup_{j\in\Z}I_A^j \;, \quad  \mathcal{I}_B=\bigcup_{j\in\Z}I_B^j \quad \text{and} \quad \mathcal{I}=\mathcal{I}_A\cap\mathcal{I}_B.$$ 
Remark that $\mathcal I\neq \emptyset$.   Also let $C_A^j=[-1/5,1/5]+j$, which is the connected component of $\R\setminus \mathcal{I}_A$ lying between $I_A^{j-1}$ and $I_A^j$, and   $C_B^j=[-1/5,1/5]+j+1/2$. Define $$\mathcal{C}_A=\bigcup_{j\in\Z}C_A^j \quad \text{and} \quad  \mathcal{C}_B=\bigcup_{j\in\Z}C_B^j \, .$$
By construction we get that $\mathcal C_A=\R\setminus \mathcal{I}_A$ and $\mathcal C_B=\R\setminus \mathcal{I}_B$. More precisely $C_A^j\subset I_B^{j-1}$ and $C_B^j\subset I_A^j$ for each $j\in \Z$. 

Let $A^+$ and $A^-$ be the positive and negative cones for the order $\preceq_A$ and define $B^+$ and $B^-$ analogously. By construction of the action \begin{equation} \label{eq union}
X^j_A:=\bigcup_{a\in A^+} \widetilde{\rho}(a)I_A^{j-1}\cup \bigcup_{a\in A^-} \widetilde{\rho}(a)I_A^{j} \subset C_A^j  
\end{equation} where the left side of the union lies in the open left half of $C^j_A$, i.e. $(j-1/5,j)$, and vice-versa. Define \begin{equation} \label{eq union2} U^j_A:=C^j_A\setminus cl(X^j_A)\;,\quad\widehat X_A^j=cl(X_A^j)\setminus X_A^j\;, \quad X_A=\bigcup_{j\in\Z}X_A^j \;,\quad \mathcal{U}_A=\bigcup_{j\in\Z}U_A^j \end{equation}
and notice that $\mathcal{U}_A$ is a union of intervals that are wandering under $\widetilde{\rho}(A)$, which are translates of the intervals mentioned in Remark \ref{rem extra wandering}. Of course, the same definitions and remarks are made for $B$. 


\vsc

\begin{prop} \label{prop ping pong}The action $\widetilde\rho:A*B \to Homeo_+(\R)$ is faithful and any point inside $\mathcal I$ has free orbit under $A*B$. Moreover the set of fixed points of any non-trivial $w\in A*B$ has empty interior, and  if  $\widetilde\rho(w)(1/2)=1/2$, then $w\in B$.
\end{prop}

\proof 

Since $\mathcal I_A=\R\setminus \mathcal C_A$ (respectively $\mathcal I_B=\R\setminus \mathcal C_B$) is a reunion of wandering intervals  separated by global fixed points of $\widetilde\rho(A)$ (respectively $\widetilde\rho(B)$), we get that 
\[ \widetilde\rho(a)(\mathcal \R \setminus \mathcal C_A)\subset \mathcal C_A \quad \text{ and } \quad\widetilde\rho(b)(\mathcal \R\setminus \mathcal C_B)\subset \mathcal C_B \]
for every $a\in A\setminus\{\id\}$ and every $b\in B\setminus \{\id\}$. Since the sets $\mathcal C_A$ and $\mathcal C_B$ are disjoint and $\mathcal{I}=\mathcal{I}_A\cap\mathcal{I}_B=\R \setminus (\mathcal C_A \cup \mathcal C_B)$ is non empty,  we get that $\mathcal C_A$, $\mathcal C_B$ is a poing pong configuration of $\widetilde\rho(A)$ and $\widetilde\rho(B)$.  By Lemma \ref{lem ping pong}, we obtain that $\widetilde\rho:A*B\to \homeo_+(\R)$ is faithful and every point in $\mathcal I$ has a free orbit under $\widetilde\rho(A*B)$.

We now show that the stabilizer of $1/2$ is the subgroup $B$. Certainly we have that  $\widetilde\rho(b)(1/2)=1/2$ for every $b\in B$.  Conversely, let  $w=a_0b_0\cdots a_nb_n\in A\ast B$ be a word where all letters are non trivial except -perthaps- the first and last one. If $n>0$ we see that $\widetilde\rho(a_nb_n)(1/2)\in \mathcal{C}_A\subset \mathcal{I}_B $. On the other hand $$\widetilde\rho(a_jb_j)\mathcal{I}_B\subset \widetilde\rho(a_j)\mathcal{C}_B\subset \widetilde\rho(a_j)\mathcal{I}_A\subset \mathcal{C}_A\subset \mathcal{I}_B $$ for every $j=1,\ldots,n-1$. As we have seen, $\widetilde{\rho}(b_0)\mathcal{I}_B$ does not contain $1/2$, so $\widetilde\rho(w)(1/2)\neq 1/2$ for $n>1$ or $n=0$ and $a_0\neq \id$, i.e. when $w\notin B$.

Finally, we check  that the set of fixed points of a non-trivial $w\in A*B$ has empty interior. For this we set $\mathcal{U}=\mathcal{U}_A\cap\mathcal{U}_B$,  prove that that the points in $\mathcal{I}\cup\mathcal{U}$ have free orbits, and show that the orbit of $\mathcal{I}\cup\mathcal{U}$ is dense. We have already established that points in $\mathcal{I}$ have free orbits. For $x\in \mathcal{U}_A$, Remark \ref{rem extra wandering} says it has a free orbit under $\widetilde\rho(A)$ contained in $\mathcal{U}_A$. On the other hand, formula (\ref{eq union}) shows that $$\widetilde\rho(b)\mathcal{U}_A\subset \widetilde\rho(b)\mathcal{I}_B \subset X_B $$ if $b\in B \setminus\{\id\}$, which is disjoint from $\mathcal{U}_A$ (since $X_B\subset \mathcal{C}_B \subset \mathcal{I}_A $). The same argument can also show that $\widetilde\rho(b)(X_A)\subset X_B$ and $\widetilde\rho(a)(X_B)\subset X_A$ for non-trivial $a\in A$, $b\in B$ (a weaker version of ping-pong configuration on $X=X_A\cup X_B$ than Lemma \ref{lem ping pong}). This implies that for $w\in A\ast B$ with $|w|\geq 2$, we have $\widetilde\rho(w)(x)\in X$ which is disjoint from $\mathcal{U}_A$. The same applies for points in $\mathcal{U}_B$ which proves that these orbits are free.

To show density, observe that we have a disjoint union \begin{equation} \label{eq descomposicion}
I_B^{j+1}=(\mathcal{I}\cap I_B^{j+1})\cup C_A^j = (\mathcal{I}\cap I_B^{j+1})\cup U^j_A \cup \left(\bigcup_{a\in A^+} \widetilde{\rho}(a)I_A^{j-1}\cup \bigcup_{a\in A^-} \widetilde{\rho}(a)I_A^{j}\right) \cup \widehat X^j_A  
 \end{equation}
where we just spell out the decomposition given by (\ref{eq union}) and (\ref{eq union2}) in a way that is ready for recursion, decomposing the intervals $I^j_A$ in the analog of (\ref{eq descomposicion}). Notice that $\widehat X_A=\bigcup_j\widehat X^j_A$ is a closed set with empty interior, and so is $\widehat X_B=\bigcup_j\widehat X^j_B$. Using the decomposition (\ref{eq descomposicion}) and its analog for $I^j_A$ recursively, we see that the complement of the orbit of $\mathcal{I}\cup\mathcal{U}$ is the orbit of $\widehat X_A\cup\widehat X_B$, which has empty interior by Baire's theorem.   $\hfill\square$

\vsc
We point out that so far our construction does not require the orders $\preceq_A$ and $\preceq_B$ to be bi-invariant: it works the same in the category of countable left-orderable groups. Thus, applying the construction in this more general setting, we obtain

\begin{cor}\label{coro left-orders} If $(A,\preceq_A)$ and $(B,\preceq_B)$ are countable left-ordered groups, then there is a left-ordering of  $A*B$ extending $\preceq_A$ and $\preceq_B$.

\end{cor}

\proof  


For $x\in \mathcal I$,  consider the ordering on $A*B$ defined by $w_1\prec w_2$ if $\widetilde \rho(w_1)(x)  < \widetilde\rho(w_2)(x)$. Since $\mathcal I\subset \mathcal I_A$ (respectively $\mathcal I\subset \mathcal I_B$) we wet that $x$ is a reference point for $\preceq_A$ (respectively $\preceq_B$) and so $\preceq$ extends $\preceq_A$ (respectively $\preceq_B$).
  $\hfill\square$

\vs

To obtain a the similar statement for the biorderable case we need to perform one last modification to $\widetilde\rho$.

\vsc
\noindent \underbar {Step 3:} {\em Cutting the  ping-pong.} 

\vsc

This modification is very simple. Recall that $0$ is a fixed point of $\widetilde\rho(A)$ and $1/2$ a fixed point of $\widetilde\rho(B)$. We just change that $\widetilde\rho$ so that now $\widetilde\rho(A)$ equals the identity on $(-\infty,0]$ (but is kept intact on $[0,\infty)$) and $\widetilde\rho(B)$ equals the identity on $(-\infty,1/2]$ (but is kept intact elsewhere). 
We call $$\rho:A*B\to Homeo_+(\R)$$ the resulting representation. Remark that the action $\rho$ has no longer free orbits. Yet, for any given $w\in A*B$ any point $x\in \R$ sufficiently large satisfies that $\rho(w)(x)=\widetilde\rho(w)(x)$ and  hence  $\rho:A*B\to Homeo_+(\R)$  is also a faithful  action.

\vsc
We now want to use the action $\rho$ to induce a biorder on $A*B$. For this, given any $w\in A*B$, we define $x_w$ as the  infimum of the support of $\rho(w)$, namely
$$x_w:=\inf \{x\in\R\mid \rho(w)(x)\neq x\}.$$
Since $\rho$ is faithful and trivial in $(-\infty,0]$, we have that $x_w$ is a well defined real number for any non-trivial $w\in A*B$. Our main observation is that on the right of $x_w$, the element $\rho(w)$ (is non-trivial and) has a {\em well defined sign}. 

\begin{prop} \label{prop principal} For any non-trivial $w\in A*B$ there is small right-neigborhood $V_w$ of $x_w$ on which $\rho(w)$ acts non-trivially and moreover  either $$x\leq \rho(w)(x) \text{ or } \rho(w)(x)\leq x$$ for every $x\in V_w$.

\end{prop}

Assuming  Proposition \ref{prop principal} we can give the proof of the first part of  Theorem \ref{teo main 2}.

\vsc

{\noindent{\em Proof of Theorem \ref{teo main 2} item $i)$:}  Define 
\begin{equation}\label{eq el cono} P=\{w\in A*B \mid \rho(w)(x)\geq x \text{ for all } x\in V_w\}.\end{equation}
It is routine to check that $P$ is a sub-semigroup of $A*B$ that is disjoint from its inverse $P^{-1}$.  Moreover,  since the condition $\rho(w)(x)\geq x$ in some right-neighborhood of $x_w$ is invariant under conjugation by any homeomorphisms of the line, we have that $P$ is normal as well. We leve details  to the reader.   Finally, the point of Proposition \ref{prop principal} is that it implies that $A*B= P\cup P^{-1}\cup \{id\}$, and so $P$ is the positive cone of a total and bi-invariant order $\preceq^*$   on $A*B$. $\hfill\square$


The proof of Proposition \ref{prop principal} is based on the fact that it holds for elements in the factors, namely:

\begin{rem} \label{rem free orbit under A} If $a\in A\setminus\{\id\}$, and $b\in B\setminus\{\id\}$ we have that \begin{itemize}
\item $x_a=0$ and $\id\prec_A a$ if and only if $x\leq \rho(a)(x)$ for all $0<x$,
\item $x_b=1/2$ and $\id\prec_B b$ if and only if $x\leq \rho(b)(x)$ for all $1/2<x$.
\end{itemize}

Indeed on one hand the facts about the infimum of the support are just the definition of $\rho$ by cutting $\widetilde\rho$.  On the other hand since   $\preceq_A$ and $\preceq_B$ are bi-orders, the rest comes from Remark \ref{rem dynamical real for bi-orders} and the facts that the constructions of blowin up and periodic repetition used for defining $\widetilde\rho_A$ and $\widetilde\rho_B$ preserve the property of Remark \ref{rem dynamical real for bi-orders}.

\end{rem}


In order to show Proposition \ref{prop principal}, note that any non-trivial element $w\in A*B$ can be written  as $w=a_0 b_n^{a_n}\cdots b_1^{a_1}$ where $a_0,\ldots,a_n\in A$, $b_1,\ldots,b_n\in B$, and $b^a:=aba^{-1}$. 
Notice that if $a_0\neq \id$, then  Proposition \ref{prop principal} holds for   true for $w$. 
Indeed since $\rho(a_0)$ is supported on $(0,\infty)$ and  the support of each $\rho(b_i^{a_i})$ is $(\rho(a_i)(\frac{1}{2}),\infty)$, we have that $x_w=0$ and that $\rho(w)$ agrees with $\rho(a_0)$ on a right-neighborhood of $0$.
Therefore for proving Proposition \ref{prop principal}  we need only to focus on the case where $a_0=\id$.   For this we introduce the concept of reduction of  a word $w\in A*B$ of the form $ b_n^{a_n}\cdots b_1^{a_1}$  at a point $x\in \R$.

\begin{dfn} \label{def reduction}
Let $w=c_n\ldots c_1$ where $c_i=a_ib_i{a_i^{-1}}$. For  $x\in\R$, denote by  $\{x_i \}_{i=0}^n$,  the {\bf trajectory} of $x$ under $w$, namely the sequence given by $x_0=x$, and $x_i=\rho(c_i\cdots c_1)(x)$ for $i=1,\ldots,n$, and define the {\bf reduction} of $w$ at $x$, denoted by $red(w,x)$, as the word obtained by removing all the $c_i$  for which $\rho(a_i^{-1})(x_{i-1})<1/2$.
\end{dfn}

For $w=c_n\ldots c_1$ as in Definition \ref{def reduction}, the condition $\rho(a_i^{-1})(x_{i-1})<1/2$ in the definition of $red(w,x)$ is precisely capturing when the trajectory of $x$ remains stationary at time $i$, that is when $x_{i-1}=x_i$.  Observe that in the case $ 1/2\leq \rho(a_i^{-1})(x_{i-1})$ both maps  $\rho(c_i)$ and $\widetilde{\rho}(c_i)$ agrees  on the interval $[x_{i-1},+\infty)$, and in particular $\rho(c_i)(x_{i-1})=\widetilde{\rho}(c_i)(x_{i-1})$.

It is clear that $red(w,x$) takes only finitely many values when $x$ varies: from $red(w,x)=id$ when $x$ is small enough (certainly $red(w,0)=id$) to $red(w,x)=w$ when $x$ is sufficiently large. In the next lemma we observe that these values  may only change  when the trajectory of $x$ steps on the points $\rho(a_1)(1/2),\ldots, \rho(a_n)(1/2)$.


\begin{lem} \label{lem reduction} For $w=c_n\ldots c_1$ where $c_i=a_ib_i{a_i^{-1}}$ and $x\in \R$, there is a  right-neighborhood $V=[x,x+\delta)$ of $x$ such that:
\begin{enumerate}
\item[$i)$] $red(w,x)=red(w,y)$ for $y\in V$.
\item[$ii)$]  $\rho(w)$ and $\rho(red(w,x))$ agree on $V$.
\item[$iii)$]  $\rho(red(w,x))$ and $\widetilde{\rho}(red(w,x))$ agree on $[x,\infty)$.
\end{enumerate}
Furthermore, if  $red(w,y)\neq red(w,x)$ for all $y<x$, then there is $i\in \{1,\ldots,n\}$ such that $x_{i-1}=\rho(a_i)(1/2)$, where $\{x_i \}_{i=0}^n$ is the trajectory of $x$ under $w$.
\end{lem}

{\em Proof:} 
By construction of $red(w,x)$, the points $\{x_i\}_{i=0}^n$ in the trajectory of $x$ under $w$ are the same that the points in the  trajectory of $x$ under $red(w,x)$. In particular we have that $\rho(w)(x)=\rho(red(w,x))(x)$. Since the trajectory of a point depends continuously of the point itself and the condition for removing a letter $c_i$ in the reduction is open, we have that there is $\delta>0$ such that $red(w,x)=red(w,y)$ for every $y\in [x,x+\delta)$, and that $\rho(w)$ and $\rho(red(w,x))$ agree on $[x,x+\delta)$. Finally, as observed above, we have that $\rho(c_i)$ agrees with $\widetilde\rho(c_i)$ on $[x_{i-1},\infty)$ whenever the letter $c_i$ is not removed in the reduction.  In particular, since $red(w,x)$ is already reduced, we have that $\widetilde \rho(red(w,x))$ agree with $\rho(red(w,x))$ on $[x_0,\infty)$.  

Now assume that $red(w,y)\neq red(w,x)$ for all $y<x$. This means that some letter $c_i$ was removed in the reduction on $y$ but not removed in the reduction on $x$. That is $\rho(a_i^{-1})(x_{i-1})\geq 1/2$ but  $\rho(a_i^{-1})(y_{i-1})<1/2$, where $\{y_i\}_{i=0}^n$ denotes the trajectory of $y$ under $w$. Since this holds for every $y<x$ we have that $\rho(a_i^{-1})(x_{i-1})=1/2$ as desired. $\hfill\Box$

\vsc

With the notion of reduction at hand, we are now in position to  find out which points can be attained as $x_w$ for some $w\in A*B$.


\begin{lem} \label{lema soporte} For $w=a_0 b_n^{a_n} \ldots b_1^{a_1}\in A*B$, we have that $x_w=0$ when $a_0\neq \id$ or $x_w$ belongs to the $\rho$-orbit of $1/2$ when $a_0=\id$.
\end{lem}


\proof  For $w=a_0 b_n^{a_n}\cdots b_1^{a_1}$ with $a_0\neq \id$ we already observed (after  Remark \ref{rem free orbit under A}) that $x_w=0$. So assume that $w=c_n\ldots c_1$ with $c_i=b_i^{a_i}$.   We claim that $x_w$ satisfies that $red(w,x_w)$ is non-trivial yet $red(w,y)$ is the trivial word for every $y<x_w$. 

Indeed on one hand  $red(w,x_w)$ is a non-triviality word since, by item $ii)$ in Lemma \ref{lem reduction}, $\rho(red(w,x_w))$ agrees with $\rho(w)$ on a right-neighborhood of $x_w$ and so its action is non-trivial.   On the other hand the triviality of $red(w,y)$ for every $y<x_w$ follows since, by item $i)$ in Lemma \ref{lem reduction}, we have that every point in a right-neighborhood of $y$ is fixed by $\rho(red(w,y))$. But by item $iii)$ in Lemma \ref{lem reduction}, this implies that $\widetilde\rho(red(w,y))$ has an interval of fixed points which, by Proposition \ref{prop ping pong}, implies that $red(w,y)$ is the trivial word.

By  the claim and the last part of Lemma \ref{lem reduction} and,  we get that some point in the trajectory of $x_w$ by  $w$ is in the $\rho$-orbit of $1/2$.  Thus $x_w$ itself is in the $\rho$-orbit of $1/2$. $\hfill\Box$




\begin{lem} \label{lema germen} Let $w\in A*B$ be such that $x_w=1/2$. Then there is a right-neighborhood $V$ of $1/2$ where $\rho(w)$ agrees with $\rho(b)$ for some $b\in B\setminus\{\id\}$.
\end{lem}

\proof Let $\overline{w}=red(w,{1/2})$. Then, by Lemma \ref{lem reduction}, we have that  $\rho(w)$ agrees with $\rho(\overline{w})$ and with $\widetilde{\rho}(\overline{w})$ in a 
right-neighborhood $V$ of $1/2$. Since $\widetilde{\rho}(\overline{w})$ fixes $1/2$, from Proposition \ref{prop ping pong},  we get that $\overline{w}\in B$. It must be non-trivial, for otherwise it would be $x_w>1/2$. $\hfill\Box$

\vsc

\noindent {\em Proof of Proposition \ref{prop principal}:}  Let $w=a_0 b_n^{a_n} \ldots b_1^{a_1}$. If $a_0\neq \id$, we already pointed out  after Remark \ref{rem free orbit under A} that $x_w=0$ and that $\rho(w)$ agrees with $\rho(a_0)$ on a right-neighborhood of $0$. 
Otherwise, by Lemma \ref{lema soporte}, we have that $x_w$ is in the orbit of $1/2$, thus it is conjugate to an element $v$ with $x_v=1/2$, and Lemma \ref{lema germen} allows us to conclude that there is $b\in B\setminus\{\id\}$ and a right-neighborhood of $x_w$ where $\rho(w)$ agrees with a conjugate of $\rho(b)$ in a right-neighborhood of $1/2$. Thus by Remark \ref{rem free orbit under A} we finish the proof.  
$\hfill\Box$ 

\vsc

\subsection{Invariance under automorphisms}
\label{sec invariance}
Let $\varphi=f\ast g$ where $f:A\to A$ and $g:B\to B$ are  automorphisms that preserves the bi-orders $\preceq_A$ and $\preceq_B$ from Theorem \ref{teo main 2}. We wish to show that the order $\preceq^*$ on $A\ast B$ constructed in item $i)$ of  Theorem \ref{teo main 2} is preserved by $\varphi$. 

For this we first show in Lemma \ref{lem functor} that the action $\widetilde\rho\circ \varphi$ is topologically-conjugated to the action $\widetilde\rho$, and then show that the same happens for $\rho\circ \varphi$ and $\rho$.   The construction of this conjugation is interesting in its own right and resonates with the constructions of \cite{bludov-glass} and \cite{rivas_higman}.

\begin{lem}\label{lem functor} There is  an increasing homeomorphism $h:\R\to\R$ that fixes  $\frac{1}{2}\Z$ pointwise and conjugates the  $\widetilde{\rho}$ action with the $\widetilde{\rho}\circ\varphi$ action, namely $$h\circ\widetilde{\rho}(w)\circ h^{-1}(x)=\widetilde{\rho}(\varphi(w))(x) \,\,\mbox{ for every }w\in A\ast B,\text{ and every }x\in \R.$$

\end{lem}

\begin{proof}

 

By Remark \ref{rem automorfismo blow} there is an increasing homeomorphism  $h_A$ that conjugates $\rho_A^{blow}$ to $\rho_A^{blow}\circ f$. Extend it by repetition to $\widetilde h_A:\R\to\R$, so that $\widetilde h_A$ equals $T_{n} \circ h_A \circ T_{n}^{-1}$ on $[n,n+1]$. Analogously, take $h_B$ conjugating $\rho_B^{blow}$ to $\rho_B^{blow}\circ g$ and extend it to $\widetilde h_B:\R\to\R$ so that on $[n, n+1]+1/2$ it equals $T_{n+1/2} \circ h_B \circ T_{n+1/2}^{-1}$.

We shall construct $h$ as the uniform limit of a sequence of maps $h_k$ which will be defined recursively. Recall the notations in the construction of $\widetilde\rho$ (Step 2 in \S \ref{sec countable}) leading to Proposition \ref{prop ping pong}. Notice that $\widetilde h_A$ fixes every point in $\mathcal{I}_A$ and preserves $C^j_A$ for every $j\in\Z$. The same goes for $B$.

First let $h_0=\widetilde h_B\circ\widetilde h_A$. Note that $\widetilde h_A$ and $\widetilde h_B$ commute and $h_0$ fixes every point in $\mathcal{I}$. Also that $h_0$ preserves $C^j_A$ and $C^j_B$ for each $j$. By the observations after Remark \ref{rem automorfismo blow} we see that $h_0$ maps $$\widetilde{\rho}(a)I_A^{j}\to \widetilde{\rho}(f(a))I_A^{j} \mbox{ for every } a\in A,a\neq 1,j\in\Z$$ $$\widetilde{\rho}(b)I_B^{j}\to \widetilde{\rho}(g(b))I_B^{j} \mbox{ for every } b\in B,b\neq 1,j\in\Z$$ by affine maps. Clearly $\frac{1}{2}\Z$ is fixed by $h_0$, and for $x\in\mathcal{I}$ and $c\in A\cup B -\{1\}$ we have $\widetilde{\rho}(\varphi(c))\circ h_0(x)=h_0\circ\widetilde{\rho}(c)(x) $.

For $w\in A\ast B$ let $|w|$ be its length as a reduced word with letters in the non-trivial elements of $A$ and $B$. Denote $W_n=\{w\in A\ast B:|w|=n\}$ and $W_{\leq n}=\{w\in A\ast B:|w|\leq n\}$. Let $W^A$ and $W ^B$ be the words with last letter in $A$ or $B$ respectively, and put $W^A_n=W^A\cap W_n$ and $W^A_{\leq n}=W^A\cap W_{\leq n}$, and the same for $B$. Consider $$\mathcal{C}_k = \left(\bigcup_{w\in W^A_k,j\in\Z}\widetilde\rho(w)C_B^j \right)\cup\left(\bigcup_{w\in W^B_k,j\in\Z}\widetilde\rho(w)C_A^j \right) $$
Using that $C_A^j\subset I_B^{j-1}$ and $C_B^j\subset I_A^j$ and (\ref{eq union}) we see that $\mathcal{C}_{k+1}\subset \mathcal{C}_k$.




We define $h_k$ recursively so that it agrees with $h_{k-1}$ on $\R\setminus \mathcal{C}_{k}$ and we detail the construction on $\mathcal{C}_{k-1}$ as follows: If $w\in W_k^A$, by induction $h_{k-1}$ takes $\widetilde\rho(w)C_B^j$ to $\widetilde\rho(\varphi(w))C_B^j$ by an affine map, that agrees with $\widetilde\rho(\varphi(w)w^{-1})$ on this interval. We define $h_k$ in $\widetilde\rho(w)C_B^j$ as $$\widetilde\rho(\varphi(w))\circ \widetilde h_B \circ \widetilde\rho(w^{-1}):\widetilde\rho(w)C_B^j\to \widetilde\rho(\varphi(w))C_B^j$$ Notice that if $b\in B\setminus\{\id\} $ then $wb\in W_{n+1}^B$ and $h_k$ takes $\widetilde\rho(wb)I_B^j$ to $\widetilde\rho(\varphi(wb))I_B^j$ by an affine map which equals $\widetilde\rho(\varphi(wb)b^{-1}w^{-1})$ on the interval, conjugating the behaviour of $\widetilde h_B$. As $C^j_A\subset I^j_B$, we get for $h_k$ what we assumed by induction for $h_{k-1}$. Reversing the roles of $A$ and $B$ we define $h_k$ on intervals of the form $\widetilde\rho(w)C_A^j$ for $w\in W_k^B$. 

The fact that $h_k$ are increasing homeomorphisms can be obtained by induction, moreover $h^{-1}_k$ is obtained from the the same construction applied to $\varphi^{-1}$. It is clear that that $\frac{1}{2}\Z$ is fixed by $h_k$: it is fixed by $h_0$ and disjoint from $\mathcal{C}_k$ for $k\geq 1$. Also, it can can be shown by induction that for $i\leq k$, we have $$\widetilde{\rho}(\varphi(w))h_k(x)=h_k(\widetilde{\rho}(w)x ) \mbox{ for } w\in W_{\leq i},\mbox{ and }x\notin \mathcal{C}_{k-j}.$$



We show now that $\{h_k\}$ is a Cauchy sequence with respect to the uniform norm. If $n<m$, then $h_n$ and $h_m$ differ only on $\mathcal{C}_n$, and for $x\in \mathcal{C}_n$ we see that $h_n(x)$ and $h_m(x)$ are in the same component of $\mathcal{C}_n$. Notice that for $a\in A-\{1\}$, the slope of $\widetilde{\rho}(a)$ on $\mathcal{I}_A$ is bounded by $1/3$, and the same is true for the slope of $\widetilde{\rho}(b)$ on $\mathcal{I}_B$ for $b\in B-\{1\}$. Thus the length of the components of $\mathcal{C}_n$ are bounded by $(2/5)(1/3)^n$ and we have $$||h_n-h_m ||_{\infty}\leq \frac{2}{5}\left(\frac{1}{3}\right)^n  \quad \mbox{for } n\leq m $$




So $h_k$ has a uniform limit $h$ that is continuous and increasing. Since $h_k^{-1}$ is also a Cauchy sequence, we have $h_k^{-1}\to h^{-1}$, so $h$ is a homeomorphism. Conditions (1) and (2) now follow directly.  
\end{proof}

\noindent {\em Proof of  Theorem \ref{teo main 2} item $ii)$}: From Lemma \ref{lem functor}, we know that there is an increasing homeomorphisms of the line $h$ that conjugates the representations $\widetilde\rho$ and $\widetilde \rho(\varphi)$  of $ A*B$.  By definition of $\rho$ in Step 3 above, we have that $\rho(A)$ agrees with $\widetilde \rho(A)$ on $[0,\infty)$ (and is trivial  elsewhere), so  $h$ conjugates $\rho(A)$ with $\rho(\varphi(A))$ since $h(0)=0$. Analogously, since $h(1/2)=1/2$ we have that  $h$ also conjugates $\rho(B)$ with $\rho(\varphi(B))$.  In particular, $h$ conjugates $\rho(A*B)$ with $\rho(\varphi(A*B))$, that is 
$$h\circ {\rho}(w)\circ h^{-1}(x)={\rho}(\varphi(w))(x)$$
for every  $w\in A*B$ and every $x\in \R$.   But as pointed out  in the proof of Theorem \ref{teo main 2} item $i)$,   the condition $\rho(w)(x)\geq x$ in some right-neighborhood of $x_w$ is invariant under conjugation by any increasing homeomorphisms of the line. In particular $\varphi(w)$ belongs to the positive cone $P$ from (\ref{eq el cono})  whenever $w$ is in $P$. This certainly implies that $\varphi$ preserves the bi-order $\preceq^*$. $\hfill\Box$




\section{The uncountable case}
\label{sec uncountable}

It is well known that (left or bi) orderability is a local property.  This principle is manifested in different theorems relating the orderability of an ambient group with the orderability of its finitely generated subgroup or even its finitely generated sub-semigroups. See the result of Burns and Hale \cite{burns-hale} for the former case and the results of Lo\'s \cite{los} and Onishi \cite{onishi} for the latter.   Burns-Hale theorem from \cite{burns-hale} has been generalize to several of other settings related to orderability, see the work of Clay \cite{clay-bh} for an account of some of them,  and the technique from Lo\'s and Onishi led the second author to obtain a local criterion for {\em not allowing an isolated ordering} \cite{rivas free}. In \cite{navas},  Navas  interpreted these classical results as a consecuence of the compactness of the space of orderings.

In this section we provide another variation of this general principle by showing that a group supports an ordering invariant under a given automorphisms if and only if all its countable subgroups supports an ordering invariant under the automorphism.

To be more precise, given a group $G$, a subgroup $H\leq G$ and an automorphisms $\varphi:G\to G$, we denote by $H^\varphi=\langle \varphi^n(H)\mid n\in \Z\rangle$ the group generated by $H$ and all its images under $\varphi$. Note that  $H^\varphi$ is countable whenever  $H$ is countable.  Clearly, every $\varphi$-invariant bi-order of $G$ induces a $\varphi$-invariant bi-order on every subgroup $H^\varphi$ of $G$. For the converse we have

\begin{thm}\label{teo invariant Burns Hale}Let $G$ be a  group, and $\varphi:G\to G$ be an automorphism.
Assume that for every countable subgroup $H\leq G$, there is a bi-order of $H^\varphi$ that is invariant under $\varphi$.  Then $G$ supports a bi-order that is inveriant under $\varphi$.
\end{thm}

\proof This follows by  compactness of $\mathcal{PBO}(G)$, the set of partial bi-orderings on $G$ (see \S \ref{sec PBO}). Indeed, for every countable $H\leq G$, consider $\chi(H^\varphi)$ to be the set of all partial bi-orderings on $G$ whose restriction to $H^\varphi$ is a total and $\varphi$-invariant  bi-ordering. The set $\chi(H^{\varphi})$ is a closed subset of $\mathcal{PBO}(G)$ which is non-empty by hypothesis. In addition,  a finite intersection
$\chi(H_1^{\varphi}) \cap \ldots \cap \chi(H_n^{\varphi})$ of these subsets is non-empty since it contains $\chi(\langle H_1, \ldots , H_n\rangle^\varphi)$, where $\langle H_1, \ldots , H_n\rangle$ is the (countable) group generated by $H_1,\ldots, H_n$. By the finite intersection property of compact sets, we get that $\chi:=\bigcap_H \chi(H^{\varphi})$,  where $H$ run over all countable subgroups of $G$, is non-empty. Happily,  any element in $\chi$ is  a total bi-order of $G$ which, moreover, is $\varphi$-invariant.   $\hfill\Box$

\vsc

With Theorem \ref{teo invariant Burns Hale} at hand, we can finish the proof of our main result.

\vsc

\noindent{\em Proof of Theorem \ref{teo main}:} Let $G=A*B$ be the free product of two bi-orderable  group $A$ and $B$, and let $H\leq A*B$ be a countable subgroup. Clearly $H$ is contained in the free product of two countable subgroup $A_{count}*B_{count}$: it is enough to take $A_{count}$ to be the group generated by all the $A$-letters in the words $w\in H$, and $B_{count}$ the analog group but for $B$-letters.   Furthermore, if $\varphi=f*g:G\to G$ is the product of the automorphisms  $f:A\to A$ and $g:B\to B$, then  
$$H^{f*g}\leq  (A_{count}*B_{count})^{f*g}= A_{count}^f * B_{count}^g.$$

Now if $\preceq_A$ and $\preceq_B$ are bi-orderings of $A$ and $B$ that are invariant under $f$ and $g$ respectively, then Theorem \ref{teo main 2} implies that $A^f_{count}*B^g_{count}$ admits a bi-orderings that is invariant under $f*g$. Hence the same holds fof $H^{f*g}$. Thus, applying  Theorem \ref{teo invariant Burns Hale} we obtain that $A*B$ supports a bi-ordering that is invariant under $f*g$. $\hfill\square$


\vs\vs

{\em  Juan Alonso} \hfill {\em Crist\'obal Rivas}

Centro de Matem\'atica, Facultad de Ciencias  \hfill Dpto. de Matemáticas

Universidad de la Rep\'ublica\hfill Universidad de Chile

Igu\'a 4225, Montevideo, Uruguay. \hfill Las Palmeras 3425, Santiago, Chile.

juan@cmat.edu.uy \hfill cristobalrivas@u.uchile.cl

\end{document}